\documentclass{amsart}
\usepackage{amsfonts}
\usepackage{amssymb}
\usepackage{amsmath}
\usepackage{amsrefs}
\usepackage{mathrsfs}
\usepackage{mathtools}

\usepackage{color}
\usepackage[dvipsnames]{xcolor}
\usepackage{tikz}
\usepackage{setspace}
\usepackage{multicol}
\usetikzlibrary{calc,intersections,decorations.pathreplacing}
\usepackage{ragged2e}
\usepackage[linktocpage=true]{hyperref}
\usepackage{pgfplots}
\usepackage{nicefrac}
\usepackage{array}

\usetikzlibrary{arrows.meta}
\usetikzlibrary{arrows}

\newtheorem{thm}{Theorem}[section]

\newtheorem{lem}[thm]{Lemma}
\newtheorem{prop}[thm]{Proposition}

\theoremstyle{definition}

\newtheorem{defn}[thm]{Definition}
\newtheorem{prop-def}[thm]{Proposition--Definition}

\newtheorem{Q}[thm]{Question}

\theoremstyle{remark}
\newtheorem{rmk}[thm]{Remark}

\newcommand{\cc}{{\mathbb C}}

\newcommand{\modu}{\mathscr}
\newcommand{\sh}{\mathcal}

\usetikzlibrary{fit}

\usetikzlibrary{fit}

\hypersetup{ colorlinks=true, citecolor=black, linkcolor=blue, urlcolor=blue}

\begin{document}

\title[Reducible CY threefolds with countably many rational curves]{Reducible Calabi-Yau threefolds with countably many rational curves}


\author[A. Zahariuc]{A. Zahariuc}
\address{Department of Mathematics, UC Davis, One Shields Ave, Davis, CA 95616}
\email{azahariuc@math.ucdavis.edu}





\begin{abstract}
We give a class of examples of reducible threefolds of CY type with two irreducible components for which (it is reasonably easy to prove that) no family of admissible genus zero stable maps sweeps out a surface, yet such stable maps occur in infinitely many degrees.
\end{abstract}

\maketitle


\section{Introduction}

A rational curve on a smooth complex projective variety $X$ is an integral curve $C \subset X$ of geometric genus zero. We say that the rational curve $C \subset X$ is \textbf{isolated} if the map $\smash{\tilde{C} \to X}$ from the normalization doesn't fit in any nonconstant family of maps $\varphi:T \times {\mathbb P}^1 \to X$ over an integral base $T$, where such a family of maps ought to be considered constant if the image of $\varphi$ is a curve.

\begin{Q}\label{hard question}
Does there exist a Calabi-Yau threefold which contains infinitely many rational curves and all its rational curves are isolated?
\end{Q}

A negative answer would be shocking because the property is expected to hold for ``most'' Calabi-Yau threefolds. However, to the best of the author's knowledge, the answer is not actually known and the proof is likely currently out of reach. In other words, it doesn't seem to be known whether any nontrivial analogues of Clemens' conjecture \cite{[Cl86]} hold, although it is known that some fail \cite{[Vo03]}.

Further motivation for studying this question comes from work of Voisin proving that Clemens' conjecture is incompatible with a longstanding conjecture of Lang \cite{[Vo04]} now widely believed to be false. It is possible that examples which confirm Question \ref{hard question} could disprove Lang's conjecture.

The main purpose of this note is to formulate a weaker version of the question above by allowing reducible threefolds and to give an affirmative answer to it.

\begin{defn}\label{reducible CY}
A reducible Calabi-Yau threefold $X$ is a reducible, reduced, connected, separated, proper scheme over ${\mathbb C}$ of dimension $3$ with simple normal crossing singularities such that

(T1) $X$ is d-semistable \cite{[Fr83]};

(T2) the dualizing sheaf $\omega_X$ is trivial;

(T3) $\mathrm{H}^1(X,{\sh O}_X)=0$.

\noindent We say that $X$ is a \emph{two-piece reducible Calabi-Yau threefold} if it consists of two smooth irreducible components $Y_1$ and $Y_2$ intersecting transversally along a smooth surface $S$.
\end{defn}

D-semistability is a necessary, but not sufficient \cite{[PP83]}, condition for the existence of smoothings. Of course, if $X$ is a two-piece reducible (Calabi-Yau) threefold, d-semistability amounts to ${\sh N}_{S/Y_1} \otimes {\sh N}_{S/Y_2} \cong {\sh O}_S$.

Even phrasing the correct analogue of Question \ref{hard question} for singular or reducible threefolds is a hard technical matter. Luckily, since our examples will only have two components, we may use the language of relative stable maps \cite{[Li01],[Li02]} to achieve that in satisfactory generality. 

Given a two-piece reducible Calabi-Yau threefold $X=Y_1 \cup_S Y_2$, let $\overline{\modu M}({\mathfrak Y}_i^\mathrm{rel},\Gamma_i)$ be the spaces of relative stable maps to $Y_i$ of topological type $\Gamma_i$. The data in $\Gamma_i$ includes an edgeless graph $G(\Gamma_i)$ with vertices $V(\Gamma_i)$ corresponding to the connected components of the relative stable maps and roots $R(\Gamma_i)$ corresponding to the distinguished marked points, as well as the data of the arithmetic genera and degrees relative to some ample line bundles of the connected components and weights $m_{\rho}$ of the roots corresponding to the multiplicities at the distinguished marked points. Denote the evaluation morphisms at the distinguished marked points by
$$ {\mathbf q}_i : \overline{\modu M}({\mathfrak Y}_i^\mathrm{rel},\Gamma_i) \longrightarrow S^{r(\Gamma_i)}, $$
where $r(\Gamma_i) = |R(\Gamma_i)|$. Two topological types $\Gamma_1$ and $\Gamma_2$ are compatible if $r(\Gamma_1)=r(\Gamma_2)$ and any two identically indexed roots in $R(\Gamma_1)$ and $R(\Gamma_2)$ (thus corresponding to distinguished marked points which are to be glued) have equal weights. 

Since we are morally concerned with limits of genus zero maps, it is understood that all compatible pairs $(\Gamma_1,\Gamma_2)$ we will consider have zero arithmetic genus on all vertices and moreover, the glued graph $G(\Gamma_1) \cup_{R(\Gamma_1) \cong R(\Gamma_2)} G(\Gamma_2)$ is a tree. 

\begin{thm}\label{main theorem}
There exists a two-piece reducible Calabi-Yau threefold $X=Y_1 \cup_S Y_2$ which contains a countable infinite number of ``admissible rational curves.'' More precisely, it satisfies the following sufficient conditions:

(T4) $S$ contains no rational curves;

(C1) no family of relative stable maps in any geometric fiber of the distinguished evaluation morphisms ${\mathbf q}_i$
sweeps out a surface on $Y_i$ for $i \in \{1,2\}$;

(C2) for any pair $(\Gamma_1,\Gamma_2)$ of \emph{compatible} topological types, the intersection of the images of ${\mathbf q}_1$ and ${\mathbf q}_2$ is a finite set of points of $S^r$, where $r=r(\Gamma_1)=r(\Gamma_2)$;

(C3) there exist infinitely many pairs $(\Gamma_1,\Gamma_2)$ of compatible topological types for which there exist $[f_i] \in \overline{\modu M}({\mathfrak Y}_i^\mathrm{rel},\Gamma_i)({\mathbb C})$ with smooth sources, $i=1,2$, such that each $f_i$ maps birationally onto its image and without contracted components into the trivial expansion $Y_i[0] = Y_i$ and ${\mathbf q}_1([f_1]) = {\mathbf q}_2([f_2])$. 
\end{thm}

In our example, $S$ is an abelian surface, so (T4) is trivially satisfied. Granted, this condition is not strictly necessary, but it allows to phrase the others in a cleaner way. Conditions (C1) and (C2) imply that all ``admissible rational curves'' are isolated and (C3) says that such ``curves'' occur infinitely often. 

Regarding (C1), note that by (T4), families of genus zero relative stable maps inside a given geometric fiber of the distinguished evaluation morphism ${\mathbf q}_i$ trivially also cannot sweep out surfaces on any additional component of an expansion of $Y_i$. Condition (C3) is not particularly elegant or interesting and perhaps could be skipped for the (somewhat open-ended) purposes of paper. However, it is required in the proof of the remark below, which justifies why the reducible threefolds in Theorem \ref{main theorem} count as having an ``infinite countable number of rational curves.''

\begin{rmk}
If $W \to B$ is a one-dimensional family of smooth Calabi-Yau threefolds degenerating at $0 \in B({\cc})$ to a two-piece reducible Calabi-Yau threefold $W_0$ which satisfies the conditions in Theorem \ref{main theorem}, then the (very) general member $W_b$ of the family contains a countable infinite number of rational curves. 

The fact that all rational curves on $W_b$ are isolated follows from (C1) and (C2). The fact that there are infinitely many such rational curves follows from (C3). Indeed, the main point is that, by rigidity of the degenerate stable maps, the fact that $\overline{\modu M}({\mathfrak W},\Gamma)$ admits an obstruction theory of the expected dimension $1$ \cite[\S\S1.2]{[Li02]} implies that rigid degenerate stable maps do deform to the smooth members of the family. Imitating the argument in the second half of the proof of \cite[Theorem 1.1]{[Za15]}, we deduce that all the nodes of the degenerate stable maps are smoothed out in this process, so the sources are all irreducible. It is also clear that no multiple covers are obtained in this way.
\end{rmk}

Of course, the reason why this doesn't answer Question \ref{hard question} is that d-semistability is not sufficient to ensure the existence of smoothings. The author does not know if any of the reducible threefolds constructed in the proof of Theorem \ref{main theorem} are actually smoothable. On one hand, since it is conjectured that there are only finitely many deformation classes of Calabi-Yau threefolds, the fact that the discrete parameter in our construction can a priori take infinitely many values could be interpreted as a red flag, though not necessarily. On the other hand, the fact that the general idea of the construction is simply to imitate that in \cite{[Za15]}, which is automatically smoothable, might be a positive sign. Also, it is likely possible to tweak the current construction in many ways. 

\section{Construction of the threefolds} 

Let $(E,q)$ be an elliptic curve with a marked point $q \in E(\cc)$. Let $S = E^2$ be the square with the projections to the two factors denoted by $\pi_1$ and $\pi_2$, diagonal $\Delta \subset S$ and second diagonal
$$ \Delta'=\{(p_1,p_2) \in S(\cc): p_1+p_2 \sim 2q\}^\mathrm{cl}, $$
where the ``cl'' superscript means closure. Note that 
\begin{equation}\label{diags} 
\Delta +\Delta' \sim 2\pi_1^{-1}(q) + 2\pi_2^{-1}(q).
\end{equation}
Indeed, the difference between the two divisors is linearly equivalent to $0$ on any geometric fiber of either projection, which implies that it is both a pullback of a divisor from $E$ via $\pi_1$ and a pullback of a class divisor from $E$ via $\pi_2$. However, this is clearly only possible if the two divisors classes are $0$.

Fix an integer $m$ and consider the line bundle
$$ {\sh L} = {\sh O}_S((m+3) \Delta - m \pi_2^{-1}(q)). $$
It is clear that $\mathrm{R}^1(\pi_1)_* {\sh L} = 0$ and that $(\pi_1)_* {\sh L}$ is a locally free sheaf of rank $3$. First, we check that
\begin{equation}\label{eq1}
\det (\pi_1)_* {\sh L}  = {\sh O}_E(-m(m+3)q).
\end{equation}
Since the relative tangent bundle of $\pi_1$ is trivial, the Grothendieck-Riemann-Roch theorem implies $\mathrm{ch}((\pi_1)_*{\sh L}) = (\pi_1)_* \mathrm{ch}({\sh L})$, hence
$$ c_1((\pi_1)_*{\sh L}) = (\pi_1)_* \frac{c_1({\sh L})^2}{2} = -m(m+3)[q] \in A_0(E), $$
confirming (\ref{eq1}) since the first Chern class of a vector bundle coincides with the first Chern class of the determinant bundle. 

Consider the projective bundle
$$ \pi:P_m := \mathbf{Proj}_E \mathrm{Sym} (\pi_1)_* {\sh L}\longrightarrow E. $$
The bundle comes equipped with a tautological invertible sheaf ${\sh O}_{P_m}(-1)$. 

\begin{lem}\label{normall}
If $L_{i,q} = \pi_i^{-1}(q)$, then
$$ {\sh N}_{S/P_m} = {\sh O}_S(3(m+3)\Delta - 3mL_{2,q} + m(m+3)L_{1,q}). $$
\end{lem}

\begin{proof}
The short exact sequence for the normal bundle of $S$ in $P_m$ implies that
$$ {\sh N}_{S/P_m} = \det ({\sh T}_{P_m}|_S). $$
The short exact sequence for the relative tangent bundle of $\pi$ gives $\det {\sh T}_{P_m} = \det {\sh T}_{\pi}$, hence ${\sh N}_{S/P_m} = \det ({\sh T}_\pi|_S)$. Moreover,
\begin{equation*}
\begin{aligned}
\det {\sh T}_\pi &= \det {\sh Hom}({\sh O}_{P_m}(-1), \pi^*[(\pi_1)_*{\sh L}]^\vee/{\sh O}_{P_m}(-1)) \\
&= {\sh O}_{P_m}(3) \otimes \pi^* \det [(\pi_1)_*{\sh L}]^\vee \\
&= {\sh O}_{P_m}(3) \otimes \pi^* {\sh O}_E(m(m+3)q)
\end{aligned}
\end{equation*}
by (\ref{eq1}), hence
\begin{equation*}
\begin{aligned}
{\sh N}_{S/P_m} &= {\sh L}^{\otimes 3} \otimes \pi_1^* {\sh O}_E(m(m+3)q) \\
&= {\sh O}_S(3(m+3)\Delta - 3mL_{2,q} + m(m+3)L_{1,q}),
\end{aligned}
\end{equation*}
as desired.
\end{proof}

\begin{lem}\label{anticanonical}
$K_P+S = 0$, that is, $S$ is an anticanonical divisor on $P$.
\end{lem}

\begin{proof}
By construction, $\{t\} \times E$ is a cubic hence an anticanonical divisor in $\pi^{-1}(t)$, so $\omega_{P/E}(S) = \pi^*{\sh M}$ for some line bundle ${\sh M} \in \mathrm{Pic}(E)$. However, $\det {\sh T}_{P_m} = \det {\sh T}_{\pi}$ as in the proof of Lemma \ref{normall}, hence $\omega_P(S) = \pi^*{\sh M}$ as well. By adjunction
$$ \pi_1^*{\sh M} = \omega_P(S)|_S =  \omega_S = {\sh O}_S, $$
hence ${\sh M} = {\sh O}_E$ and $\omega_P(S) = {\sh O}_P$, as desired.
\end{proof}

Assume that $m=2k$, where $k$ is a positive integer. Let $D_1$ be a divisor on $S \cong S_1$ consisting of very general\footnote{The ``very general'' assumption is actually \emph{not} essential. It is only used in \S\S3.2 and even there it is likely avoidable at the expense of further complications.} translates of $\Delta$ such that $D_1 \sim (2k^2+12k+18) \Delta$ and $D_2 \subset S_2$ consisting of very general translates of $\Delta'$ such that $D_2 \sim 2k^2\Delta'$. Let
$$ Y_i:= \text{Blowup}_{D_i} P_{2k} $$
for $i=1,2$. Let $S_i \cong S$ be the proper transform of $S$ and $\varphi_i:Y_i \to E$ the natural morphism whose fibers are rational surfaces.  

We construct the reducible threefold $X$ to be the transversal gluing of the disjoint threefolds $Y_1$ and $Y_2$ along $S_1 \cong S_2$ in such a way that the induced automorphism of $S \cong E^2$ is the involution which exchanges the two factors. One possible reference for gluing along closed subschemes is \cite[Corollary 3.9]{[Sch05]}. By a slight abuse of notation, we will continue to denote the intersection $Y_1 \cap Y_2$ by $S$. The notation for the projections to the two factors respects the notation in $Y_1$ and is opposite to that in $Y_2$. We will verify that $X$ satisfies the conditions of Theorem \ref{main theorem}.

\begin{prop}
The reducible threefold $X$ is a two-piece reducible Calabi-Yau threefold in the sense of Definition \ref{reducible CY}.
\end{prop}

\begin{proof} We simply need to verify conditions (T1)--(T3).

(T1) A straightforward calculation using (\ref{diags}) and Lemma \ref{normall} shows that
$$ {\sh N}_{S/P_{2k}}^{\otimes 2} \cong {\sh O}_{S}(D_1 + D_2), $$
hence the d-semistability condition is verified since ${\sh N}_{S/Y_i} = {\sh N}_{S/P_{2k}}(-D_i)$ and $D_2$ is invariant under the involution of $S$ exchanging the factors. 

(T2) This follows from Lemma \ref{anticanonical} and the fact that the proper transform of an anticanonical divisor remains anticanonical if the center of the blowup was contained inside the former divisor.

(T3) By the usual short exact sequence, it suffices to verify that the map 
$$ {\mathrm H}^*({\sh O}_{Y_1}) \oplus {\mathrm H}^*({\sh O}_{Y_2}) \longrightarrow {\mathrm H}^*({\sh O}_S) $$
is surjective in degree zero, which is trivial, and injective in degree one. The latter is a consequence of the following easy facts: the diagram
\begin{center}
\begin{tikzpicture}

\node (a) at (0,0) {${\mathrm H}^1({\sh O}_E)$};
\node (b) at (3,0) {${\mathrm H}^1({\sh O}_S)$};
\node (c) at (3,1.5) {${\mathrm H}^1({\sh O}_{Y_i})$};

\draw [->, thin] (a) -- (b);
\draw [->, thin] (a) -- (c);
\draw [->, thin] (c) -- (b);

\node at (1.5, 0.3) {$\pi_i^*$};
\node at (1.5, 1.1) {$\varphi_i^*$};

\end{tikzpicture}
\end{center}
is commutative, $\varphi_i^*$ is an isomorphism, essentially because the cohomology of the structure sheaf is a birational invariant, and $\pi_1^* \oplus \pi_2^*$ is an isomorphism. \end{proof}

\section{The genus zero stable maps}

\subsection{The rigidity requirement} Of course, the key basic observation is that, because there are no nonconstant maps from a rational curve to a genus one curve, all connected components of the relative stable maps are sent to (expansions of the) fibers of the maps $\varphi_i$.

Condition (C1) is proved completely analogously to \cite[Proposition 3.2]{[Za15]}. However, we point out that the argument ultimately is entirely equivalent to applying the following elementary lemma.

\begin{lem}
If $E \subset W$ is a smooth connected divisor on a smooth rational projective surface $W$ such that $E + K_W \sim 0$, then the image-cycles of maps from rational trees into $W$ which intersect $E$ in any predetermined divisor do not cover the surface $W$ birationally. 
\end{lem}

\begin{proof} 
If such a family exists, consider a semistable completion over a proper base. Properness mandates that some image-cycle of a map in the complete family contains $E$, but this is clearly impossible.
\end{proof}

Assume by way of contradiction that (C2) fails. Let ${\mathbf q}$ be the fiber product of ${\mathbf q}_1$ and ${\mathbf q}_2$ over their common target. Then we may find a nonsingular connected affine curve $ \iota:Z \to  \overline{\modu M}({\mathfrak Y}_1^\mathrm{rel},\Gamma_1) \times_{S^r} \overline{\modu M}({\mathfrak Y}_2^\mathrm{rel},\Gamma_2)$ such that ${\mathbf q} \iota$ is not constant. For each $v \in V:=V(\Gamma_1) \sqcup V(\Gamma_2)$, let
$$ \lambda_v:Z \longrightarrow E $$
be the function which specifies the fiber of $\varphi_i$ containing the connected component of the source of $(f_i)_z$ corresponding to the vertex $v \in V(\Gamma_i)$. Moreover, by further shrinking $Z$ if necessary, we may assume that all $(f_i)_z$ map to the same expansion $Y_i[n_i]_0$ of $Y_i$ for $i=1,2$. 

For two vertices $v$ and $w$ which become adjacent after gluing the graphs in $\Gamma_1$ and $\Gamma_2$, denote by $[vw$ the root at $v$ ``pointing towards'' $w$. Let $R(v)$ be the set of all roots incident to $v$. Let
\begin{equation*}
\begin{aligned} 
R_0(v) & := \left\{ \rho \in R(v): (\pi_{\rho} {\mathbf q} \iota)(z) \notin D_1 \cup D_2 \right\} \\
R_+(v) & := \left\{ \rho \in R(v): (\pi_{\rho} {\mathbf q} \iota)(z) \in D_j \backslash D_i \right\} \\
\end{aligned}
\end{equation*}
where $\pi_\rho:S^r \to S$ is the projection corresponding to the root $\rho$, and let $R_{\geq 0}$ be their union. Shrinking $Z$ even more, we may assume that the properties whether $(\pi_{\rho} {\mathbf q} \iota)(z) \in D_1$ or $D_2$ or neither hold either for all $z \in Z$ or for none. 

Assume that $v \in V(\Gamma_i)$. Let $d_v \geq 0$ be the degree of the pullback of ${\sh O}_{P_{2k}}(1)$ to the connected component corresponding to $v $. Let $\smash{ L_v = \pi_i^{-1}(\lambda_v(z)) \subset S }$ and $\smash{ P_v = \pi^{-1}(\lambda_v(z))}$ living inside $P_{2k}$. Let $C_v \subset P_v$ be the image-cycle of the restriction of the map $f_i$ to the connected component corresponding to the vertex $v$, so that $C_v$ has degree $d_v$. For each $\alpha \in L_v$, let 
$$ m_{v,\alpha} = (C_v \cdot L_v)_{\alpha,P_v}, $$
the intersection multiplicity at $\alpha$. If $\alpha \in L_v \cap D_i$, let $\hat{\lambda}_{v,\alpha}(z) = \pi_j(\alpha) \in E$, where $\{j\} = \{1,2\} \backslash \{i\}$. 

\begin{lem}\label{linear equivalence} The following linear equivalence holds
\begin{equation}\label{main}
\sum_{[vw \in R_{\geq 0}(v)} m_{[vw} \lambda_w(z) + \sum_{ \alpha \in L_v \cap D_i} m_{v,\alpha}\hat{\lambda}_{v,\alpha}(z) \sim_{\mathrm{rat}} d_v((2k+3)\lambda_v(z) - 2kq). 
\end{equation}
In particular, the sum of the coefficients on the left is $3d_v$.
\end{lem}

\begin{proof}
The basic idea is that the hyperplane divisor class of $P_v$ restricts on $L_v \cong E$ to ${\sh O}_E((2k+3)\lambda_v(z) - 2kq)$. Let $p_{[vw} \in L_v \subset S$ be the image of the distinguished marked point corresponding to the root $[vw$. Because all bubbles of the expansions are ${\mathbb P}^1$-bundles over $S$, condition (T4) and the admissibility condition \emph{along the chain of bubbles} imply that for each closed point $\alpha \in L_v \backslash D_i$,
$$ \sum_{p_{[vw} = \alpha} m_{[vw} = m_{v,\alpha}  $$ 
and (\ref{main}) follows after summing over all $\alpha \in L_v$. \end{proof}

Fix $z \in Z(\cc)$ and a nonzero tangent vector $\mathbf{u}$ to $Z$ at $z$. For each $v \in V$, let
$$ \lambda_v'=\frac{\mathrm{d}\lambda_v}{\mathrm{d}z} := \mathrm{d}\lambda_v({\mathbf u}) \in \cc, $$
where a trivialization of ${\sh T}_E$ has been fixed beforehand. The derivative of $\smash{ \hat{\lambda}_{v,\alpha}(z) }$ can be defined similarly. However, note that $\smash{ \hat{\lambda}'_{v,\alpha} \in \{\pm \lambda'_v\} }$ and also $\lambda'_w \in \{\pm \lambda'_v\}$ if $[vw \in R_+(v)$. Taking derivatives in (\ref{main}), we obtain
\begin{equation*}
d_v(2k+3)\frac{\mathrm{d}\lambda_v}{\mathrm{d}z} = \sum_{[vw \in R_{\geq 0}(v)} m_{[vw} \frac{\mathrm{d}\lambda_w}{\mathrm{d}z} + \sum_{ \alpha \in L_v \cap D_i} m_{v,\alpha}\frac{\mathrm{d}\hat{\lambda}_{v,\alpha}}{\mathrm{d}z}.
\end{equation*}
The key idea is to choose $v \in V$ for which $|\lambda'_v|$ is \emph{maximal}. It is possible to choose such a $v$ such that $d_v \neq 0$ thanks to the kissing condition, the connectivity of the glued graph and the fact that $D_1 \cap D_2$ is finite. By the triangle inequality,
\begin{equation*}
\begin{aligned}
d_v(2k+3) \left| \frac{\mathrm{d}\lambda_v}{\mathrm{d}z} \right| & \leq \sum_{[vw \in R_{\geq 0}(v)} m_{[vw} \left| \frac{\mathrm{d}\lambda_w}{\mathrm{d}z} \right| + \sum_{ \alpha \in L_v \cap D_i} m_{v,\alpha} \left| \frac{\mathrm{d}\hat{\lambda}_{v,\alpha}}{\mathrm{d}z} \right| \\
& \leq \left[ \sum_{[vw \in R_{\geq 0}(v)} m_{[vw} + \sum_{ \alpha \in L_v \cap D_i} m_{v,\alpha} \right] \left| \frac{\mathrm{d}\lambda_v}{\mathrm{d}z} \right| = 3d_v \left| \frac{\mathrm{d}\lambda_v}{\mathrm{d}z} \right|,
\end{aligned}
\end{equation*}
hence $\lambda'_v = 0$ for the chosen $v$ and thus $\lambda'_v = 0$ for all $v$. Then all $\lambda_v$ are constant which contradicts the assumption that ${\mathbf q} \iota$ is not constant. This completes the proof that (C2) is satisfied.

\subsection{The existence requirement} Finally, we need to check that condition (C3) holds. There is a lot of leeway in constructing the desired pairs of relative stable maps. Below is one possibility. 

We start by introducing some overdue notation. For simplicity, write $N_1 = 2k^2+12k+18$ and $N_2 = 2k^2$. Let $D_1 = \Delta_1+\Delta_2+...+\Delta_{N_1}$ and $D_2 = \Delta'_1+\Delta'_2+...+\Delta'_{N_2}$,
\begin{equation*}
\begin{aligned}
\Delta_j &=\{(p_1,p_2) \in S(\cc): p_1-p_2 \sim q_j-q\}^\mathrm{cl}, \\
\Delta'_j &=\{(p_1,p_2) \in S(\cc): p_1+p_2 \sim q'_j+q\}^\mathrm{cl}.
\end{aligned}
\end{equation*}
All that we need from the assumption that the translates are very general is that
\begin{equation}\label{sole}
q_1+q_2+...+q_{N_1} \sim_E N_1q \text{ and } q'_1+q'_2+...+q'_{N_2} \sim_E N_2q
\end{equation}
and their combinations are the only nontrivial relations involving the $q$s.

We will construct a chain of lines in $X$. More accurately, we desire to find curves $\ell_1,\ell_3,...,\ell_{n-1} \subset Y_1$ and $\ell_2,\ell_4,...,\ell_n \subset Y_2$ and points $p_1,p_2,...,p_{n-1} \in S$ satisfying the following properties:

$\bullet$ $p_j \in \ell_j \cap \ell_{j+1}$ for all $j$;

$\bullet$ $\ell_1$ is the proper transform under the blowup $Y_1 \to P_{2k}$ of a line inside some fiber of $\pi$ and $\ell_1$ intersects the exceptional divisors corresponding to $\Delta_1$ and $\Delta_2$;

$\bullet$ $\ell_j$ is the proper transform under the blowup $Y_i \to P_{2k}$ of a line inside some fiber of $\pi$ and $\ell_j$ intersects the exceptional divisor corresponding to $\Delta_3$ or $\Delta'_3$, for $2 \leq j \leq n-1$;   

$\bullet$ $\ell_n$ is the proper transform under the blowup $Y_2 \to P_{2k}$ of a line inside some fiber of $\pi$ and $\ell_n$ intersects the exceptional divisors corresponding to $\Delta'_1$ and $\Delta'_2$;

$\bullet$ all $\ell_j$ live inside different fibers of $\varphi_1$ and $\varphi_2$. 
\newline The only closed conditions come from Lemma \ref{linear equivalence}, which in this case reads
\begin{equation}\label{mess}
(2k+3) \lambda_j - 2kq \sim_E
\begin{cases}
2\lambda_1 + \lambda_2 + 2q -q_1-q_2 & \text{if } j=1, \\
\lambda_{j-1} -\lambda_j + \lambda_{j+1} +q + q'_3 & \text{if } j \leq n-2 \text{ even}, \\
\lambda_{j-1} +\lambda_j + \lambda_{j+1} +q - q_3 & \text{if } j \geq 3 \text{ odd}, \\
\lambda_{n-1} - 2\lambda_n + q'_1+ q'_2 +2q & \text{if } j = n. \\
\end{cases}
\end{equation}
Let $f_1$ and $f_2$ be the inclusions
\begin{equation*}
\begin{aligned}
f_1 &: \ell_1 \sqcup \ell_3 \sqcup \ell_5 \sqcup ... \hookrightarrow Y_1 \\
f_2 &: \ell_2 \sqcup \ell_4 \sqcup \ell_6 \sqcup ... \hookrightarrow Y_2.
\end{aligned}
\end{equation*}
The fact that the construction above is sound, which amounts to saying that the points $p_j$ don't belong to the exceptional divisors of the blowups $Y_i \to P_{2k}$, and that the conditions in (C3) are satisfied turns out to boil down to the following straightforward computation due to the uniqueness of the relations (\ref{sole}): the $n \times n$ symmetric tridiagonal matrix 
$$\begin{bmatrix*}[c]
2k+1 & -1 & 0  & 0 & 0 & \cdots & 0 & 0 & 0 \\
-1 & 2k+4 & -1 & 0 & 0 & \cdots & 0 & 0 & 0 \\
0 & -1 & 2k+2 & -1 & 0 & \cdots & 0 & 0 & 0 \\
0 & 0 & -1 & 2k+4 & -1 & \cdots & 0 & 0 & 0 \\
0 & 0 & 0 & -1 & 2k+2 & \cdots & 0 & 0 & 0 \\
\vdots  & \vdots  & \vdots  & \vdots  & \vdots  & \ddots & \vdots  & \vdots  & \vdots \\
0 & 0 & 0  & 0 & 0 & \cdots & 2k+4 & -1 & 0 \\
0 & 0 & 0  & 0 & 0 & \cdots & -1 & 2k+2 & -1 \\
0 & 0 & 0  & 0 & 0 & \cdots & 0 & -1 & 2k+5 \\
\end{bmatrix*}
$$
is invertible and all the entries in the first row/column of the inverse matrix are pairwise distinct. The calculation is uninteresting and as such will be skipped.

\bigskip

\noindent \textit{Acknowledgements.} This paper is morally related to the earlier work \cite{[Za15]} and I would like to reiterate my thanks to Joe Harris, Xi Chen and the others acknowledged in that paper. I would also like to thank Brian Osserman for many valuable suggestions and Andreas Knutsen for useful discussions.

\end{document}